\documentclass[12pt]{article}
\usepackage{amssymb,amsfonts,amsthm}
\usepackage{amstext}

\usepackage{amsmath}

\usepackage{color}
\usepackage[latin1]{inputenc}
\usepackage[active]{srcltx}
\def\openC{{\rm C\kern-.18cm\vrule width.8pt height 7pt depth-.2pt \kern.18cm}}
\def\openN{{{\rm I}\kern-.16em {\rm N}}}
\def\openR{{{\rm I}\kern-.16em {\rm R}}}
\def\openT{{{\rm T}\kern-.42em {\rm T}}}
\def\openZ{{{\rm Z}\kern-.28em{\rm Z}}}

\newtheorem{thm}{Theorem}[section]
\newtheorem{cor}[thm]{Corollary}

\theoremstyle{definition}

\newtheorem{rem}[thm]{Remark}
\newtheorem{ex}[thm]{Example}

\begin{document}

\title{
{\textbf{{Matrix valued positive definite kernels related to the generalized Aitken's integral for Gaussians}}} \vspace{-4pt}
}

\author{\sc
V. A. Menegatto \& C. P. Oliveira }
\date{}
\maketitle \vspace{-30pt}
\bigskip

\begin{center}
\parbox{13 cm}{{\small We introduce a method to construct general multivariate positive definite kernels on a nonempty set $X$ that employs a prescribed bounded completely monotone function
and special multivariate functions on $X$.\ The method is consistent with a generalized version of Aitken's integral formula for Gaussians.\
In the case where $X$ is a cartesian product, the method produces nonseparable positive definite kernels that may be useful in multivariate interpolation.\ In addition, it can be interpreted as an abstract multivariate generalization of the well-established Gneiting's model for constructing space-time covariances commonly cited in the literature.\ Many parametric models discussed in statistics can be interpreted as particular cases of the method.}}
\end{center}

\thispagestyle{empty}

\section{Introduction}
\label{sec1}
Let $X$ be a nonempty set and write $M_q(\mathbb{C})$ to denote the set of all $q\times q$ matrices with complex entries.\ A kernel $K=[K_{m,n}]_{m,n=1}^q: X \times X \to M_q(\mathbb{C})$ is {\em positive definite} if  for every positive integer $N$ at most the cardinality of $X$ and distinct points $x_1, \ldots, x_N$ in $X$, the block matrix $[[K_{m,n}(x_\mu,x_\nu)]_{\mu,\nu=1}^N]_{m,n=1}^q$ of order $Nq$ is positive semi-definite, that is,
\begin{equation}\label{PDBAS}
\sum_{\mu,\nu=1}^N c_\mu^{*} K(x_\mu,x_\nu)c_\nu =\sum_{m,n=1}^q \sum_{\mu,\nu=1}^N \overline{c_\mu^m} c_\nu^n K_{m,n}(x_\mu,x_\nu)\geq 0,
\end{equation}
whenever $c_1, \ldots, c_N$ are column vectors in $\mathbb{C}^q$ and $c_\mu=[c_\mu^1 \,\ldots \,c_\mu^q]^\intercal$.\ The star notation refers to conjugate transposition of column vectors in $\mathbb{C}^q$.\ If the matrices
$[[K_{m,n}(x_\mu,x_\nu)]_{\mu,\nu=1}^N]_{m,n=1}^q$ are all
positive definite, that is, the inequalities in (\ref{PDBAS}) are strict when at least one of the vectors
$c_\mu$ is nonzero, then the positive definite kernel $K$ is termed {\em strictly positive definite} on $X$.\ The two classes of kernels introduced above will be denoted by
$PD_q(X)$ and $SPD_q(X)$, respectively.\ Kernels in these classes correspond to the standard positive definite kernels studied in \cite{berg} when we set $q=1$ and identify $M_q(\mathbb{C})$ with $\mathbb{C}$.\ The importance of matrix valued positive definite kernels in their various formats may be ratified in the references \cite{alfonsi,micheli, minh, wittwar}.
Examples of kernels in $PD_q(X)$ and $SPD_q(X)$ can be easily constructed.\ If $A$ is a positive semi-definite matrix in $M_q(\mathbb{C})$, then the constant kernel
$$
K(x,x')=A, \quad x,x'\in X,
$$
belongs to $PD_q(X)$.\ If $f_1, \ldots, f_q$ are kernels in $PD_1(X)$, then the kernel $K$ given by the formula
$$
K(x,x')=\mbox{Diag}(f_1(x,x'),\ldots, f_q(x,x')), \quad x,x' \in X,
$$
belongs to $PD_q(X)$.\ Further, if all the $f_m$ belong to $SPD_1(X)$, then $K$ belongs to $SPD_q(X)$.\ Moving the other way around, if $K$ is a kernel in $PD_q(X)$ and $c\in \mathbb{C}^q$, then $f(x,x') = c^* K(x,x')c$, $x,x'\in X$, defines a function in $PD_1(X)$.\ If $c\neq 0$ and $K$ belongs to $SPD_q(X)$, then $f$ actually belongs to $SPD_1(X)$.

The purpose of this paper is to introduce methods to construct abstract matrix-valued mappings with the additional requirement of positive definiteness and strict positive definiteness.\ In many cases, the method yields flexible models, once it encompasses models found in geophysical sciences, including probabilistic weather forecasting, data assimilation, statistical analysis of climate model output, etc, when one makes the right choice for $X$ and set a metric structure in it.

The method itself will be based on bounded completely monotone functions and special matrix valued functions attached to the notion of conditional negative definiteness.\ Recall that the {\em complete monotonicity} of a function $f:(0,\infty) \to \mathbb{R}$ is characterized by two properties: $f$ is $C^{\infty}$ and $(-1)^nf^{(n)}(t)\geq 0$ for $n=0,1,\ldots$ and $t\in (0,\infty)$.\ Throughout the paper, we will not distinguish between a bounded completely monotone function and its unique continuous extension to $[0,\infty)$.\ A kernel $K=[K_{m,n}]_{m,n=1}^q: X \times X \to M_q(\mathbb{C})$  is {\em conditionally negative definite} if it is Hermitian and the matrices $[[K_{m,n}(x_\mu,x_\nu)]_{\mu,\nu=1}^N]_{m,n=1}^q$ are of {\em negative type}, that is, the quadratic forms (\ref{PDBAS}) are nonpositive whenever the vectors $c_\mu$ satisfy $\sum_{\mu=1}^N c_\mu =0$.\ The conditionally negative definite kernel $K$ is {\em strictly conditionally negative definite} if the matrices $[[K_{m,n}(x_\mu,x_\nu)]_{\mu,\nu=1}^N]_{m,n=1}^q$ are of strict negative type for $N\geq 2$, that is, the quadratic forms are negative whenever $N\geq 2$ and at least one $c_\mu$ is nonzero.\ These two classes of kernels will be denoted by $CND_q(X)$ and $SCND_q(X)$, respectively.\ Examples of kernels in $CND_1(X)$ and $SCND_1(X)$ can be found in \cite{berg} while connections between the classes $PD_1(X)$ and $CND_1(X)$ are described in \cite{belton,berg,kapil}.\ As for examples in the classes $CND_q(X)$ and $SCND_q(X)$ one may imitate the procedures adopted for producing kernels in $PD_q(X)$ and $SPD_q(X)$ previously mentioned.

All the major results we intend to prove here will be based on a generalization of Aitken's integral formula for computing Gaussians: if $A$ is a positive definite matrix in $M_q(\mathbb{R})$ and $b$ is a vector in $\mathbb{R}^q$, then
$$\int_{\mathbb{R}^q} e^{\displaystyle{-u^\intercal  A u+i\,b^\intercal u}}du = \frac{\pi^{q/2}}{\sqrt{\det\,A }} e^{\displaystyle{-b^\intercal  (4A)^{-1} b}}.
$$
Aitken's integral itself corresponds to the formula above in the case $b=0$.\ A proof for the generalized Aitken's integral formula can be reached by mimicking the proof of Aitken's integral in \cite[p. 340]{searle} but an  independent proof is available in \cite{meneoli}.\ This reference also contains univariate versions of some of the results to be described here.

Before we proceed to the outline of the paper, it is worth mentioning that if $X$ is actually a cartesian product of sets, the method to be presented here lead to nonseparable kernels, a desirable property in applications.\ Meanwhile, in some specific cases, the method will become a generalization of the well established Gneiting's contribution in \cite{gneiting} on the construction of kernels in $PD_1(\mathbb{R}^q\times \mathbb{R}^d)$.\ Gneiting's classical result is as follows: for a bounded completely monotone function $\phi: (0,\infty)\to\mathbb{R}$ and a positive valued function $f$ with a completely monotone derivative, it asserts that the formula
\begin{equation}\label{mother}
G_r((x,y),(x',y'))=\frac{1}{f(\|y-y'\|^2)^{r}}\phi\left(\frac{\|x-x'\|^2}{f(\|y-y'\|^2)}\right),\quad x,x'\in \mathbb{R}^q;\,y,y'\in \mathbb{R}^{d},
\end{equation}
defines a kernel $G_r$ in $PD_1(\mathbb{R}^q \times \mathbb{R}^d)$, whenever $r\geq d/2$ and $\|\cdot\|$ denotes the usual norms in both $\mathbb{R}^q$ and $\mathbb{R}^d$.\ The boundedness of $\phi$ is required in order to make $\phi(0^+)<\infty$.\ The references \cite{mene0,porcu01} include some extensions and generalizations of this important result along with additional references on the topic.

The paper proceeds as follows.\ Section 2 begins with the description of two additional notions to be employed in the paper, one for families of vector functions and another for families of matrix functions, along with examples.\ The first major result of the paper is Theorem \ref{p-detem}: it describes a  method to construct kernels in $PD_p(Y)$ from bounded completely monotone functions, special families of vector functions on $Y$ and special families of matrix functions on $Y$.\ Further, it provides a sufficient condition in order that the resulting kernel be in $SPD_p(Y)$.\ At the end of the section we discuss some examples and detach a relevant consequence of Theorem \ref{p-detem}.\ The main result in Section 3 expands Theorem \ref{p-detem} via integration with respect to a convenient measure.\ We separate a special simpler version of the theorem in Corollary \ref{conse}.\ Section 4 describes extensions of Theorems \ref{p-detem} and \ref{p-detemen} that lead to kernels in $PD_p(X\times Y)$.\ Applications and a multivariate abstract extension of the classical Gneiting's result are described.

\section{The main result for positive definiteness on a single set}
\label{sec2}
This section contains the first main contribution in the paper to be made explicit in Theorem \ref{p-detem}.\ It provides a method to construct functions in $PD_q(Y)$ using completely monotonic functions via Aitken's formula.\ A sufficient condition for strict positive definiteness is included.\ The contribution itself demands two notions for families of functions with domain $Y$  which we now discuss.

For a matrix function $G$ in $CND_q(Y)$ and a vector $u$ from $\mathbb{C}^q$, the kernel
$$
(y,y') \in Y\times Y \mapsto u^* G(y,y') u
$$
belongs to $CND_1(Y)$.\ Further, the kernel belongs to $SCND_1(Y)$ whenever $G$ belongs to $SCND_q(Y)$ and $u$ is nonzero.\
Theorem \ref{p-detem} will demand a family $\{G_{m,n}: m,n=1,\ldots,p\}$ for which all the matrix kernels
$$
(y,y')\in Y \times Y\mapsto \left[u^{\intercal} G_{m,n}(y,y') u\right]_{m,n=1}^p, \quad u\in \mathbb{R}^q,
$$
belong to $CND_p(Y)$.\ Since this is not easily achievable, the following example is apposite.

\begin{ex} \label{exa0} Define
		$$
		G_{m,n}(y,y')= g_m(y) + g_n(y'), \quad y,y'\in X,
		$$
		where the $g_m : Y \to M_q(\mathbb{C})$ are functions subject to our choice.\ If $y_1,\ldots, y_N$ are distinct points in $Y$, $c_1,\ldots, c_N$ are vectors in $\mathbb{C}^p$ such that
		$\sum_{\mu=1}^Nc_\mu=0$, and $u\in \mathbb{C}^q$, then
		\begin{eqnarray*}
			\sum_{\mu,\nu=1}^N c_\mu^* \left[u^{\intercal} G_{m,n}(y_\mu,y_\nu)u\right]_{m,n=1}^p c_\nu & = & \sum_{n=1}^p \sum_{\nu=1}^N  c_\nu^n \sum_{\mu=1}^N \sum_{m=1}^p \overline{c_\mu^m} u^{\intercal} g_m(y_\mu) u\\
			& & \hspace*{3mm}+\sum_{m=1}^p \sum_{\mu=1}^N  \overline{c_\mu^m} \sum_{\nu=1}^N \sum_{n=1}^p c_\nu^n u^{\intercal} g_n(y_\nu) u =0,
		\end{eqnarray*}
		that is, the matrix function
		$$
		(y,y')\in X \times X\mapsto \left[u^{\intercal} G_{m,n}(y,y') u\right]_{m,n=1}^p,
		$$
		belongs to $CND_p(Y)$.
\end{ex}

\begin{ex}  Set  $G_{m,n}=0$ when $m\neq n$ and pick each $G_{m,m}$ in the class $CND_q(Y)$.\ Keeping the $c_\mu$ and the $y_\mu$ as in Example \ref{exa0}, it is easily seen that
		$$
		\sum_{\mu,\nu=1}^N c_\mu^* \left[u^{\intercal} G_{m,n}(y_\mu,y_\nu)u\right]_{m,n=1}^p c_\nu=\sum_{m=1}^p \sum_{\mu,\nu=1}^N \overline{c_\mu^m}c_\nu^m u^{\intercal} G_{m,m}(y_\mu,y_\nu)u\leq 0.
		$$
Thus, the matrix function
		$$
		(y,y')\in X \times X\mapsto \left[u^{\intercal} G_{m,n}(y,y') u\right]_{m,n=1}^p,
		$$
		belongs to $CND_p(Y)$.
\end{ex}

Theorem \ref{p-detem} will also need special families $\{H_{m,n}: m,n=1,\ldots,p\}$ of vector functions $H_{m,n}: Y\times Y \to \mathbb{C}^q$.\ As matter of fact, we will require families
for which all the matrix functions
$$
(y,y')\in Y\times Y \mapsto \left[e^{\displaystyle{i\,H_{m,n}(y,y')^* u}}\right]_{m,n=1}^p,\quad u\in \mathbb{R}^q,
$$
belong to $PD_p(Y)$.\ Again, this is not easy to achieve, reason why a simple example is handy.

\begin{ex}{\rm Let us set
		$$
		H_{m,n}(y,y')=h_m(y)-h_n(y'),\quad y,y' \in Y,
		$$
		where $h_m:Y \to \mathbb{R}^q$, $m=1,\ldots, p$.\ If $y_1, \ldots, y_N$ are distinct points in $Y$ and $c_1, \ldots, c_N$ are vectors in $\mathbb{C}^p$, then
		$$
		\sum_{\mu,\nu=1}^N c_\mu^* \left[e^{\displaystyle{i\,H_{m,n}(y_\mu, y_\nu)^\intercal  u}}\right]_{m,n=1}^p c_\nu = \left|\sum_{\mu=1}^N \sum_{m=1}^p \overline{c_\mu^m}\, e^{\displaystyle{i\,h_m(y_\mu)^\intercal u}}\right|^2 \geq 0, \quad u\in \mathbb{R}^q,
		$$
		that is, the kernels
		$$
		(y,y')\in Y\times Y \mapsto \left[e^{\displaystyle{i\,H_{m,n}(y,y')^* u}}\right]_{m,n=1}^p,\quad u\in \mathbb{R}^q,
		$$
		belong to $PD_p(Y)$.}
\end{ex}

We observe that if the matrix functions
$$
(y,y')\in Y\times Y \mapsto \left[e^{\displaystyle{i\,H_{m,n}(y,y')^* u}}\right]_{m,n=1}^p,\quad u\in \mathbb{R}^q,
$$
belong to $PD_p(Y)$, then each $H_{m,n}$ must be anti-symmetric in the sense that
$$
\mbox{Re\,}H_{m,n}(y,y')=-\mbox{Re\,}H_{m,n}(y',y), \quad y,y' \in Y.
$$
In particular,
$$
\mbox{Re\,}H_{m,n}(y,y)=0, \quad m,n=1,\ldots,p;\, y\in Y.
$$

Some specific properties of Hadamard exponentials will be needed.\ We recall that if $A$ is a matrix in $M_q(\mathbb{C})$, then its {\em Hadamard exponential} is the
matrix
$$
e^{\circ A}:=[e^{A_{\mu\nu}}]_{\mu,\nu=1}^q.
$$
Let $M_q(\mathbb{R})$ denote subset of $M_q(\mathbb{C})$ composed of real matrices only.\ If $A \in M_q(\mathbb{R})$ is symmetric and of negative type, then the Hadamard exponential of $-A$ is positive semi-definite.\ It is positive definite if, and only if,
$$
A_{\mu\mu}+A_{\nu\nu} < 2A_{\mu\nu}, \quad \mu\neq \nu.
$$
These facts are proved in Lemma 2.5 in \cite{reams} albeit \cite{mene} analyzed similar properties earlier.\ As an obvious consequence, we have that if $A \in M_q(\mathbb{R})$ is of strict negative type, then the Hadamard exponential of $-A$ is positive definite.\
Finally, if a real symmetric block matrix $A=[[A_{mn}(\mu\nu)]_{\mu,\nu=1}^N]_{m,n=1}^q$ is of negative type, then the Hadamard exponential of $-A$ is positive definite if, and only if,
\begin{equation}\label{blockexp}
A_{mm}(\mu\mu)+A_{nn}(\nu\nu)<2A_{mn}(\mu\nu), \quad |m-n|+|\mu-\nu|>0.
\end{equation}

Below, we will use the symbol $\bullet$ to denote the Schur product of two matrices of same size.

\begin{thm} \label{p-detem}
	Let $\phi$ be a bounded and completely monotone function.\
	For each $m,n$ in $\{1,\ldots, p\}$, let $G_{m,n}:Y \times Y \to M_q(\mathbb{R})$ be a matrix function with range containing positive definite matrices only and $H_{m,n}: Y \times Y \to \mathbb{R}^q$ a vector function.\ Assume the matrix functions
	$$
	(y,y')\in Y \times Y \mapsto [u^{\intercal} G_{m,n}(y,y')u]_{m,n=1}^p,\quad u \in \mathbb{R}^q,
	$$
	belong to $CND_{p}(Y)$ and that
	$$
	(y,y')\in Y\times Y \mapsto \left[e^{\displaystyle{i\,H_{m,n}(y,y')^{\intercal} u}}\right]_{m,n=1}^p,\quad u\in \mathbb{R}^q,
	$$
	belong to $PD_p(Y)$.\ The following assertions hold for the  kernel $K: Y\times Y \to M_p(\mathbb{R})$
	given by the formula
	$$
	K(y,y')=\left[\frac{\phi\left(H_{m,n}(y,y')^\intercal  G_{m,n}(y,y')^{-1} H_{m,n}(y,y')\right)}{\sqrt{\det G_{m,n}(y,y')}} \right]_{m,n=1}^p, \quad y,y' \in Y.
	$$
	\begin{itemize}
		\item[$(i)$] $K$ belongs to $PD_p(Y)$.
		\item[$(ii)$] If $\phi$ is not identically 0 and there exists an open subset $U$ of $\mathbb{R}^q\setminus\{0\}$ so that
		$$
		u^\intercal [G_{m,m}(y,y)+G_{n,n}(y',y')-2G_{m,n}(y,y')]u<0,\quad (m,y)\neq (n,y');\, u\in U,
		$$ then $K$ belongs to $SPD_p(Y)$.
	\end{itemize}
\end{thm}
\begin{proof} We begin proving Assertion $(i)$ in the case where $\phi$ is a constant function, that is, the case in which
	$$
	K(y,y')=\left[\frac{\phi(0)}{\sqrt{\det G_{m,n}(y,y')}}\right]_{m,n=1}^p, \quad y,y' \in Y.
	$$
	Since each matrix $G_{m,n}(y,y')$ is positive definite, we may apply Aitken's integral formula to obtain
	\begin{equation} \label{key1}
	K(y,y')=\frac{\phi(0)}{\pi^{q/2}}\left[\int_{\mathbb{R}^q} e^{\displaystyle{-u^{\intercal} G_{m,n}(y,y')u}}du \right]_{m,n=1}^p, \quad y,y' \in Y.
	\end{equation}
	If $y_1, \ldots, y_N$ are distinct points in $Y$ and $c_1, \ldots, c_N$ are vectors in $\mathbb{R}^p$, then
	\begin{eqnarray*}
		\sum_{\mu,\nu=1}^N c_\mu^{\intercal}K(y_\mu,y_\nu)c_\nu & = & \frac{\phi(0)}{\pi^{q/2}}\sum_{\mu,\nu=1}^N \sum_{m,n=1}^p c_\mu^mc_\nu^n \int_{\mathbb{R}^q} e^{\displaystyle{-u^{\intercal} G_{m,n}(y_\mu,y_\nu)u}}du\\
		& = & \frac{\phi(0)}{\pi^{q/2}}\int_{\mathbb{R}^q} \sum_{\mu,\nu=1}^N c_\mu^{\intercal} E_u(y_\mu,y_\nu)c_\nu du,
	\end{eqnarray*}
	where
	$$
	E_u(y,y')=\left[e^{\displaystyle{-u^{\intercal} G_{m,n}(y,y')u}}\right]_{m,n=1}^p=e^{\displaystyle{\circ [-u^{\intercal} G_{m,n}(y,y')u}]_{m,n=1}^p},\quad u\in \mathbb{R}^q.
	$$
	One of the assumptions on the $G_{m,n}$ now yields that
	$$
	\sum_{\mu,\nu=1}^N c_\mu^{\intercal} E_u(y_\mu,y_\nu)c_\nu \geq 0, \quad u \in \mathbb{R}^q,
	$$
	and Assertion $(i)$ follows in this case.\ In the general case, the Bernstein-Widder Theorem (\cite[p. 3]{schilling}) implies that
	$$
	K(y,y')=\left[\frac{1}{\sqrt{\det G_{m,n}(y,y')}} \int_{[0,\infty)}e^{\displaystyle{-H_{m,n}(y,y')^\intercal  G_{m,n}(y,y')^{-1} H_{m,n}(y,y')\,s}}d\sigma(s)\right]_{m,n=1}^p
	$$
	for some finite and positive measure $\sigma$ on $[0,\infty)$.\ On the other hand, the generalized Aitken's integral formula provides the alternative representation
	$$
	K(y,y')=\left[\frac{1}{\pi^{q/2}}\int_{[0,\infty)}\left(\int_{\mathbb{R}^q} e^{\displaystyle{-u^{\intercal}G_{m,n}(y,y')u}} e^{\displaystyle{2i\, \sqrt{s} H_{m,n}(y,y')^\intercal u}}du\right) d\sigma(s)\right]_{m,n=1}^p.
	$$
	If the $y_\mu$ are as before and the $c_\mu$ are now complex vectors, the quadratic form
	$$
	Q:=\sum_{\mu,\nu=1}^N c_\mu^{*}K(y_\mu,y_\nu)c_\nu
	$$
	becomes
	\begin{eqnarray*}
		Q & = & \frac{1}{\pi^{q/2}} \sum_{\mu,\nu=1}^N \sum_{m,n=1}^p \overline{c_\mu^m} c_\nu^n \int_{[0,\infty)} \int_{\mathbb{R}^q} e^{-\displaystyle{u^{\intercal} G_{m,n}(y_\mu,y_\nu)u}}e^{\displaystyle{i\,2\sqrt{s} H_{m,n}(y_\mu,y_\nu)^{\intercal} u}} du d\sigma(s)\\
		& = & \frac{1}{\pi^{q/2}}\int_{[0,\infty)}\int_{\mathbb{R}^q} \sum_{\mu,\nu=1}^N c_\mu^{*} \left[E_u(y_\mu,y_\nu) \bullet E_u^s(y_\mu,y_\nu)\right] c_\nu\, du d\sigma(s),
	\end{eqnarray*}
	where
	$$
	E_u^s(y,y')=\left[e^{\displaystyle{i\,2\sqrt{s} H_{m,n}(y,y')^{\intercal} u} }\right]_{m,n=1}^p, \quad y,y' \in Y; s\geq 0.
	$$
	The assumption on the $H_{m,n}$ settles the positive semi-definiteness
	of each matrix $E_u^s(y_\mu,y_\nu)$ while the Schur Product Theorem ratifies the positive semi-definiteness of each Schur product $E_u(y_\mu,y_\nu) \bullet E_u^s(y_\mu,y_\nu)$.\ These
	arguments validate the inequality $Q\geq 0$.\\
	Let us keep the notation used above to prove Assertion $(ii)$.\ Assume further that the $c_\mu$ are not all zero vectors.\ If there exists an open subset $U$ of $\mathbb{R}^q\setminus\{0\}$ so that
	$$
	u^\intercal [G_{m,m}(y,y)+G_{n,n}(y',y')-2G_{m,n}(y,y')]u<0,\quad (m,y)\neq (n,y');\, u\in U,
	$$
	we can infer via (\ref{blockexp}) that the block matrix
	$$
	E_u(y_\mu,y_\nu)=\left[ \left[e^{\displaystyle{-u^\intercal G_{m,n}(y_\mu,y_\nu)u}}\right]_{\mu,\nu=1}^N\right]_{m,n=1}^p
	$$
	is positive definite whenever $u\in U$.\ Thus, if $\phi$ is constant and not identically 0,
	then $Q>0$ by Formula (\ref{key1}).\ If $\phi$ is nonconstant, first we invoke our assumption on the $H_{m,n}$ in order to see that
	the diagonal entries in each block matrix
	$$
	E_u^s(y_\mu,y_\nu)=\left[ \left[e^{\displaystyle{i\,2 \sqrt{s}H_{m,n}(y_\mu,y_\nu)^\intercal u}}\right]_{\mu,\nu=1}^N\right]_{m,n=1}^p
	$$
	are all equal to 1.\ An application of Oppenheim's inequality (\cite[p. 509]{horn}) shows that the Schur product $E_u(y_\mu,y_\nu) \bullet E_u^s(y_\mu,y_\nu)$ is positive definite for $u\in U$ and $s\geq 0$.\ In particular,
	$$
	\int_{\mathbb{R}^q}\sum_{\mu,\nu=1}^N c_\mu^{*} \left[E_u(y_\mu,y_\nu) \bullet E_u^s(y_\mu,y_\nu)\right] c_\nu\, du>0, \quad s\geq 0.
	$$
	Since $\sigma$ is not the zero measure we may go one step further and infer that $Q>0$.
\end{proof}


\begin{rem}
	Theorem 17 in \cite{schlather} is a very special case of Theorem \ref{p-detem}-$(i)$.
\end{rem}

Next, we present some examples that illustrate our findings.

\begin{ex}\label{exe1}  For $m=1,\ldots,p$, let $g_m: Y \to M_q(\mathbb{R})$ be a function with range containing positive definite matrices only and $h_m: Y \to \mathbb{R}^q$ an arbitrary function.\
		Setting $G_{m,n}(y,y')=g_m(y)+g_n(y')$, $y,y'\in Y$, and $H_{m,n}(y,y')=h_m(y)-h_n(y')$, $y,y'\in Y$, the assumptions in Theorem \ref{p-detem} are satisfied.\ Thus, the formula
		$$\left[\frac{\phi\left((h_m(y)-h_n(y'))^\intercal  (g_m(y)+g_n(y'))^{-1} (h_m(y)-h_n(y'))\right)}{\sqrt{\det[(g_m(y)+g_n(y')]}}\right]_{m,n=1}^p, \quad y,y'\in Y,
		$$
		defines a kernel in $PD_p(Y)$ whenever $\phi$ is bounded completely monotone function.\ The inequalities in Theorem \ref{p-detem}-$(ii)$ cannot be matched in this abstract example.
\end{ex}

\begin{ex}\label{exdi} For $m,n=1,\ldots,p$, let us set
$$
G_{m,n}(y,y')=g_{m,n}(y,y')I_q,\quad y,y'\in Y,
$$
where each $g_{m,n}$ is a positive valued kernel on $Y$ and $(y,y')\in Y\times Y \mapsto [g_{m,n}(y,y')]_{m,n=1}^p$ belongs to $CND_p(Y)$.\ Observe that
for each $m$ and $n$,
$$
u^\intercal G_{m,n}(y,y') u=\|u\|^2 g_{m,n}(y,y'),\quad u\in \mathbb{R}^q; y,y'\in Y.
$$
On the other hand, if $c_1, \ldots, c_N$ are column vectors satisfying $\sum_{\mu=1}^n c_\mu=0$ and $y_1, \ldots, y_n$ belong to $Y$, then
$$
\sum_{m,n=1}^p \sum_{\mu,\nu=1}^N \overline{c_\mu^m}c_\nu^n u^\intercal G_{m,n}(y_\mu,y_\nu)u=\|u\|^2 \sum_{m,n=1}^p \sum_{\mu,\nu=1}^N \overline{c_\mu^m}c_\nu^n g_{m,n}(y_\mu,y_\nu)\leq 0.
$$
that is, each kernel
$$
(y,y')\in Y \times Y \mapsto [u^\intercal G_{m,n}(y,y')u ]_{m,n=1}^p,
$$
belongs to $CND_p(Y)$.\ If the $H_{m,n}$ satisfy the assumptions of Theorem \ref{p-detem}, then it is promptly seen that the formula
$$
K(y,y')=\left[\frac{1}{ g_{m,n}(y,y')^{q/2}} \phi\left( \frac{\|H_{m,n}(y,y')\|^2}{g_{m,n}(y,y')}\right)\right]_{m,n=1}^p, \quad y,y' \in Y,
$$
defines a matrix kernel in $PD_p(Y)$ whenever $\phi$ is a bounded completely monotone function.
\end{ex}

\begin{ex}{\label{exdi2}\rm If we take $H_{m,n}$ as in Example \ref{exe1}, then the kernel $K$ in Example \ref{exdi} takes the form
		$$
		K(y,y')=\left[\frac{1}{g_{m,n}(y,y')^{q/2}} \phi\left( \frac{\|h_m(y)-h_n(y')\|^2}{g_{m,n}(y,y')}\right)\right]_{m,n=1}^p, \quad y,y' \in Y.
		$$
		This example has an structure that resembles that of Gneiting's model in \cite{gneiting} for the construction of space-time
		covariances.\ We can get even closer by setting $g_{m,n}:=g$ for all $m$ and $n$, where $g: Y \to (0,\infty)$ belongs to $CND_1(Y)$, a choice that leads to
		$$
		K(y,y')=\frac{1}{g(y,y')^{q/2}}\left[ \phi\left( \frac{\|h_m(y)-h_n(y')\|^2}{g(y,y')}\right)\right]_{m,n=1}^p, \quad y,y' \in Y.
		$$}
\end{ex}

The setting adopted in both Examples \ref{exdi} and \ref{exdi2} is a particular case of that detached in Theorem \ref{coro} below.\
Needless to say that the theorem can be interpreted as a multivariate version of the Gneiting's criterion in \cite{gneiting}.

\begin{thm}\label{coro}
	Let $\phi$ be a bounded and completely monotone function.\ Let $g$ be a positive valued kernel in $CND_1(Y)$ and for each $m,n$ in $\{1,\ldots, p\}$, define
	$$
	G_{m,n}(y,y')= g(y,y')I_q, \quad y,y' \in Y.
	$$
	If $H_{m,n}: Y \times Y \to \mathbb{R}^q$ is a vector function such that the matrix functions
	$$
	(y,y')\in Y\times Y \mapsto \left[e^{\displaystyle{i\,H_{m,n}(y,y')^\intercal u}}\right]_{m,n=1}^p,\quad u\in \mathbb{R}^q,
	$$
	belong to $PD_p(Y)$, then the following assertions hold for the  kernel $K: Y\times Y \to M_p(\mathbb{R})$
	given by the formula
	$$
	K(y,y')=\frac{1}{g(y,y')^{q/2}}\left[ \phi\left( \frac{\|H_{m,n}(y,y')\|^2}{g(y,y')}\right)\right]_{m,n=1}^p, \quad y,y' \in Y.
	$$
	\begin{itemize}
		\item[$(i)$] $K$ belongs to $PD_p(Y)$.
		\item[$(ii)$] If $\phi$ is not identically 0 and $g(y,y)+g(y',y')-2g(y,y')<0$ for $y\neq y'$, then $K$ belongs to $SPD_p(Y)$.
	\end{itemize}
\end{thm}


\section{An extension of the main result via integration}

Here, we extend the results proved in Section 2 by introducing a scale mixture in the formula that defines the positive definite kernels.

Our first contribution here is as follows.

\begin{thm} \label{p-detemen}
	Let $\rho$ be a nonzero positive measure on $(0,\infty)$ and $\phi$ a bounded and completely monotone function.\ For each $m,n$ in $\{1,\ldots, p\}$, let $G_{m,n}:Y \times Y \to M_q(\mathbb{R})$ be a matrix function with range containing positive definite matrices only, $H_{m,n}: Y \times Y \to \mathbb{R}^q$ a vector function and $\{P_{m,n}^s\}_{s>0}$ a family of kernels on $Y$ such that each function $s\in (0,\infty) \mapsto P_{m,n}^s(y,y')$ is $\rho$-integrable.\ If the matrix functions
	$$
	(y,y')\in Y \times Y \mapsto [u^{\intercal} G_{m,n}(y,y')u]_{m,n=1}^p,\quad u \in \mathbb{R}^q,
	$$
	$$
	(y,y')\in Y\times Y \mapsto \left[e^{\displaystyle{i\,H_{m,n}(y,y')^\intercal u}}\right]_{m,n=1}^p,\quad u\in \mathbb{R}^q,
	$$
	and
	$$
	(y,y')\in Y\times Y \mapsto [P_{m,n}^s(y,y')]_{m,n=1}^p, \quad s>0,
	$$
	belong to $CND_{p}(Y)$, $PD_p(Y)$, and $PD_p(Y)$, respectively, then the  kernel $K: Y\times Y \to M_p(\mathbb{R})$
	given by the formula
	\begin{eqnarray*}
		K(y,y')& = & \left[\frac{1}{\sqrt{\det G_{m,n}(y,y')}} \right. \\
		& \times  & \left. \int_{(0,\infty)} \phi\left(H_{m,n}(y,y')^\intercal  G_{m,n}(y,y')^{-1} H_{m,n}(y,y')\, s\right) P_{m,n}^s(y,y')d\rho(s)\right]_{m,n=1}^p
	\end{eqnarray*}
	belongs to $PD_p(Y)$.
\end{thm}
\begin{proof}
	Let $y_1, \ldots, y_N$ be distinct points in $Y$, $c_1, \ldots, c_N$ vectors in $\mathbb{C}^p$ and set $Q:=\sum_{\mu,\nu=1}^N c_\mu^{*}K(y_\mu,y_\nu)c_\nu$.\ Direct calculation shows that
	$$
	Q=\int_{(0,\infty)}\sum_{\mu,\nu=1}^N c_\mu^{*} \left[I^s(y_\mu,y_\nu) \bullet P^s(y_\mu,y_\nu)\right] c_\nu d\rho(s),
	$$
	where
	$$
	I^s(y,y')=\left[\frac{ \phi\left(\sqrt{s}H_{m,n}(y,y')^\intercal  G_{m,n}(y,y')^{-1} \sqrt{s}H_{m,n}(y,y')\right)}{\sqrt{\det G_{m,n}(y,y')}} \right]_{m,n=1}^p
	$$
	and
	$$
	P^s(y,y')=\left[P_{m,n}^s(y,y')\right]_{m,n=1}^p,\quad y,y' \in Y; s>0.
	$$
	As in the proof of Theorem \ref{p-detem}, the matrix functions
	$$
	(y,y')\in Y\times Y \mapsto \left[e^{\displaystyle{i\,\sqrt{s}H_{m,n}(y,y')^\intercal u}}\right]_{m,n=1}^p,\quad u\in \mathbb{R}^q; s>0,
	$$
	belong to $PD_p(Y)$.\ However, since the assumptions on the $G_{m,n}$ are the same as those in Theorem \ref{p-detem}, we can apply Theorem \ref{p-detem}-$(i)$ in order to see that each matrix $I^s(y_\mu,y_\nu)$
	is positive semi-definite.\ As for $P^s(y_\mu,y_\nu)$, $s>0$, they are positive semi-definite as well by our assumption on the family $\{P_{m,n}^s\}_{s>0}$.\ Thus, the Schur Product Theorem implies that
	$$
	\sum_{\mu,\nu=1}^N c_\mu^{*} \left[I^s(y_\mu,y_\nu) \bullet P^s(y_\mu,y_\nu)\right] c_\nu \geq 0,\quad s>0.
	$$
	Therefore, $Q\geq 0$.
\end{proof}

As for strict positive definiteness, the following consequence of Theorem \ref{p-detemen} holds.

\begin{thm}
	If $\phi$ is not identically zero, then the following additional assertions hold for the kernel $K$ defined in Theorem \ref{p-detemen}:
	\begin{itemize}
		\item[$(i)$]
		If there exists an open subset $A$ of $\mathbb{R}^q\setminus\{0\}$ so that
		$$
		u^\intercal [G_{m,m}(y,y)+G_{n,n}(y',y')-2G_{m,n}(y,y')]u<0,\quad (m,y)\neq (n,y');\, u\in U,
		$$ and a $\rho$-measurable subset $A$ of $(0,\infty)$ so that $\rho(A)>0$ and
		$$
		P_{m,m}^s(y,y)>0,\quad m\in \{1,\ldots,p\};\, y \in Y;\, s\in A,
		$$ then $K$ belongs to $SPD_p(Y)$.
		\item[$(ii)$] If there exists a $\rho$-measurable subset $A$ of $(0,\infty)$ so that $\rho(A)>0$ and
		$$
		(y,y')\in Y\times Y \mapsto [P_{m,n}^s(y,y')]_{m,n=1}^p \in SPD_p(Y), \quad s\in A,
		$$
		then $K$ belongs to $SPD_p(Y)$.
	\end{itemize}
\end{thm}
\begin{proof}
	Let the $x_\mu$ and the $c_\mu$ be as in the proof of Theorem \ref{p-detemen}.\ Further, assume at least one $c_\mu$ is nonzero.\ If the assumptions in $(i)$ hold, then Theorem \ref{p-detem}-$(ii)$ implies that each matrix $I^s(y_\mu,y_\nu)$ is positive definite while the diagonal entries in
	$P^s(y_\mu,y_\nu)$ are all positive for $s\in A$.\ Therefore, by Oppenheim's inequality, we can assert that
	$$
	\sum_{\mu,\nu=1}^N c_\mu^{*} \left[I^s(y_\mu,y_\nu) \bullet P^s(y_\mu,y_\nu)\right] c_\nu > 0, \quad s\in A.
	$$
	Since the measure $\rho$ is nonzero, $Q>0$.\ If the assumptions in $(ii)$ hold, we may reach the very same conclusion once the diagonal elements in $I^s(y_\mu,y_\nu)$, $s>0$,
	are given by
	$$
	\frac{\phi(0)}{\sqrt{\det G_{m,m}(y_\mu,y_\mu)}} > 0, \quad m=1,\ldots, p; \mu=1,\dots, q.
	$$
	Indeed, Oppenheim's inequality once again would imply that $Q>0$.
\end{proof}

A specially chosen family $\{G_{m,n}:m,n=1,\ldots, p\}$ in Theorem \ref{p-detemen} leads to the following improved abstract multivariate version of Gneiting's criterion in \cite{gneiting}.

\begin{cor} \label{conse}
	Let $\phi : (0,\infty) \to \mathbb{R}$ be a bounded and completely monotone function.\ For $m,n=1,2,\ldots,p$, set $G_{m,n}(y,y')=g_{m,n}(y,y')I_q$, $y,y'\in Y$, where each $g_{m,n}$ is a positive valued kernel in $CND_1(Y)$,
	let $H_{m,n}: Y \times Y \to \mathbb{R}^q$ be a vector function and $\{P_{m,n}^s\}_{s>0}$ a family of kernels on $Y$ such that each function $s\in (0,\infty) \mapsto P_{m,n}^s(y,y')$ is $\rho$-integrable.\ If the matrix functions
	$$
	(y,y')\in Y \times Y \mapsto [u^{\intercal} G_{m,n}(y,y')u]_{m,n=1}^p,\quad u \in \mathbb{R}^q,
	$$
	$$
	(y,y')\in Y\times Y \mapsto \left[e^{\displaystyle{i\,H_{m,n}(y,y')^\intercal u}}\right]_{m,n=1}^p,\quad u\in \mathbb{R}^q,
	$$
	and
	$$
	(y,y')\in Y\times Y \mapsto [P_{m,n}^s(y,y')]_{m,n=1}^p, \quad s>0,
	$$
	belong to $CND_{p}(Y)$, $PD_p(Y)$, and $PD_p(Y)$, respectively, then the kernel $K: Y\times Y \to M_p(\mathbb{R})$
	given by the formula
	$$
	K(y,y')=\left[\frac{1}{g_{m,m}(y,y')^{q/2}}\int_0^\infty \phi\left(\frac{\|H_{m,n}(y,y')\|^2s}{g_{m,n}(y,y')}\right)P_{m,n}^s(y,y')d\rho(s)\right]_{m,n=1}^p
	$$
	belongs to $PD_p(Y)$.\ Further, if $\phi$ is not identically 0, the following two additional assertions hold:
	\begin{itemize}
		\item[$(i)$] If $g_{m,m}(y,y)+g_{n,n}(y',y')-2g_{m,n}(y,y')<0$ when $(m,y)\neq (n,y')$, and there exists a $\rho$-measurable subset $A$ of $(0,\infty)$ so that $\rho(A)>0$ and
		$$
		P_{m,m}^s(y,y)>0, \quad m\in \{1,\ldots,p\};\, y \in Y;\, s\in A,
		$$
		then $K$ belongs to $SPD_p(Y)$.
		\item[$(ii)$] If there exists a $\rho$-measurable subset $A$ of $(0,\infty)$ so that $\rho(A)>0$ and
		$$
		(y,y')\in Y\times Y \mapsto [P_{m,n}^s(y,y')]_{m,n=1}^p,
		$$
		belongs to $SPD_p(Y)$ for $s\in A$, then $K$ belongs to $SPD_p(Y)$.
	\end{itemize}	
\end{cor}

\section{The main results in the case of a product of sets}

An easy way to construct kernels in $PD_p(X \times Y)$ is given by the product of a kernel in $PD_p(X)$
with another one in $PD_p(Y)$, a fact that can be ratified via the Schur Product Theorem.\ The separable kernels produced by this method
may be not suitable if one needs strong interactions between $X$ and $Y$.\ The main result in this section will provide a version of Theorem \ref{p-detem}
that leads to kernels in $PD_p(X \times Y)$ and, except for very particular cases, the kernels produced by this version will be nonseparable.\ In particular, the
aforementioned interactions are possible.\
The result explains, from a mathematical point of view, some important practical models adopted in the statistical literature.\ The proofs will be omitted once
they are very similar to those of the theorems proved in Sections 2 and 3.

\begin{thm} \label{p-detemprod}
	Let $\phi : (0,\infty) \to \mathbb{R}$ be a bounded and completely monotone function.\ For each $m,n$ in $\{1,\ldots, p\}$, let $G_{m,n}:Y \times Y \to M_q(\mathbb{R})$ be
	a matrix function with range containing positive definite matrices only and $H_{m,n}: X \times X \to \mathbb{R}^q$ a vector function.\ If the matrix functions
	$$
	(y,y')\in Y \times Y \mapsto [u^{\intercal} G_{m,n}(y,y')u]_{m,n=1}^p,\quad u \in \mathbb{R}^q,
	$$
	belong to $CND_{p}(Y)$ and
	$$
	(x,x')\in X\times X \mapsto \left[e^{\displaystyle{i\,H_{m,n}(x,x')^{\intercal} u}}\right]_{m,n=1}^p,\quad u\in \mathbb{R}^q,
	$$
	belong to $PD_p(X)$, then the kernel
	$K: (X\times Y)^2 \to M_p(\mathbb{R})$ given by
	$$
	K((x,y),(x',y'))=\left[\frac{\phi\left(H_{m,n}(x,x')^\intercal  G_{m,n}(y,y')^{-1} H_{m,n}(x,x') \right)}{\sqrt{\det G_{m,n}(y,y')}} \right]_{m,n=1}^p
	$$
	belongs to $PD_p(X \times Y)$.
\end{thm}

In the Example \ref{exa} below, we illustrate Theorem \ref{p-detemprod} in the case $X=\mathbb{R}$ and $Y=S^d$, the unit sphere in $\mathbb{R}^{d+1}$.\

\begin{ex} Define $H_{m,n}(x,x')=h_m(x)-h_n(x')$, $x,x'\in \mathbb{R}$, where each $h_m: \mathbb{R} \to \mathbb{R}^q$ is an arbitrary function.\ If $\delta$ denotes the geodesic distance in $S^d$, set
$$
		G_{m,n}(y,y')=[m+n+\delta(y,y')]I_q, \quad y,y'\in S^d.
		$$
		It is well known that $(y,y')\in S^d \times S^d \mapsto \delta(y,y')$ belongs to $CND_1(S^d)$ (see Section 4 in \cite{alex}).\ Hence, each $G_{m,n}$ has range containing positive definite matrices only.\
		On the other hand, according to Examples \ref{exdi} and \ref{exdi2},
		each kernel
		$$
		(y,y')\in Y \times S^d\mapsto [u^\intercal G_{m,n}(y,y')u ]_{m,n=1}^p, \quad u\in \mathbb{R}^q,
		$$
		belongs to $CND_p(S^d)$.\ It follows that
		$$
		K((x,y),(x',y')) = \left[\frac{1}{[m+n+\delta(y,y')]^{q/2}}\phi\left(\frac{\|h_m(x)-h_n(x')\|^2}{m+n+\delta(y,y')}\right) \right]_{m,n=1}^p
		$$
		belongs to $PD_p(\mathbb{R}\times S^d)$.\ The choice
		$$
		h_m(x)=(x,0,\ldots,0)^\intercal,\quad x\in \mathbb{R};\, m=1,\ldots,p,
		$$
		leads to the simpler example
		$$
		K((x,y),(x',y')) = \left[\frac{1}{[m+n+\delta(y,y')]^{q/2}}\phi\left(\frac{(x-x')^2}{m+n+\delta(y,y')}\right) \right]_{m,n=1}^p
		$$
		in $PD_p(\mathbb{R}\times S^d)$.
\end{ex}

A version of Theorem \ref{p-detemen} for kernels acting on the product $X \times Y$ is as follows.

\begin{thm} \label{pp-detemen}
	Let $\rho$ be a nonzero positive measure on $(0,\infty)$ and $\phi$ a bounded and completely monotone function.\ For each $m,n$ in $\{1,\ldots, p\}$, let $G_{m,n}:Y \times Y \to M_q(\mathbb{R})$ be a matrix function with range containing positive definite matrices only, $H_{m,n}: X \times X \to \mathbb{R}^q$ vector functions and $\{P_{m,n}^s\}_{s>0}$ a family of kernels on $X\times Y$ such that each function $s\in (0,\infty) \mapsto P_{m,n}^s((x,y),(x',y'))$ is $\rho$-integrable.\ If the matrix functions
	$$
	(y,y')\in Y \times Y \mapsto [u^{\intercal} G_{m,n}(y,y')u]_{m,n=1}^p,\quad u \in \mathbb{R}^q,
	$$
	$$
	(x,x')\in X\times X \mapsto \left[e^{\displaystyle{i\,H_{m,n}(x,x')^\intercal u}}\right]_{m,n=1}^p,\quad u\in \mathbb{R}^q,
	$$
	and
	$$
	((x,y),(x',y'))\in Y\times Y \mapsto [P_{m,n}^s((x,y),(x',y'))]_{m,n=1}^p, \quad s>0,
	$$
	belong to $CND_{p}(Y)$, $PD_p(Y)$, and $PD_p(X\times Y)$, respectively, then the matrix kernel $K=[K_{m,n}]_{m,n=1}^p: (X\times Y)^2 \to M_p(\mathbb{R})$ given by the formula
	\begin{eqnarray*}
		K_{m,n}((x,y) \hspace{-3mm}&,&\hspace{-3mm}(x',y')) = \frac{1}{\sqrt{\det G_{m,n}(y,y')}}  \\
		& & \hspace*{-10mm}\times \int_{(0,\infty)} \phi\left(H_{m,n}(x,x')^\intercal  G_{m,n}(y,y')^{-1} H_{m,n}(x,x')\, s\right) P_{m,n}^s((x,y),(x',y'))d\rho(s)
	\end{eqnarray*}
	belongs to $PD_p(X\times Y)$.
\end{thm}

We now move to some specific applications of Theorem \ref{pp-detemen}.

\begin{ex}\label{porcufurado}
Here we will employ the formula deduced in Theorem 1.1 in \cite{cho}:
		$$
		\mathcal{M}_\nu(r\sqrt{u})=\frac{r^{2\nu}}{2^{2\nu} \Gamma(\nu)}\int_0^\infty e^{\displaystyle{-s\;u}} e^{\displaystyle{-r^2/4s}}s^{-\nu-1}ds, \quad r,u>0.
		$$
		that defines the so-called {\em Mat\'{e}rn function}.\ This function is studied in details in \cite{cho}.\
		We may apply Theorem \ref{p-detemen} with $\phi(u)=\exp(-u)$, $u>0$ and $d\rho(s)=e^{-r^2/4s} s^{-1}ds$.\ If for $x,x'\in X$ and $y,y' \in Y$ we set
		$$
		2v_{m,n}((x,y),(x',y')):=v_m(x,y)+v_n(x',y'),
		$$
		where $v_m: X\times Y \to (0,\infty)$, for all $m$, and
		$$
		P^s_{m,n}((x,y),(x',y')):= \frac{r^{\displaystyle{2v_{m,n}((x,y),(x',y'))}} s^{\displaystyle{-v_{m,n}((x,y),(x', y'))}}}{2^{\displaystyle{2v_{m,n}((x,y),(x',y'))}} },
		$$
		for $x,x'\in X$ and $y,y'\in Y$, it is easily seen that the kernels
		$$
		((x,y),(x',y'))\in (X\times Y)^2 \mapsto \left[P_{m,n}^s((x,x'),(y,y'))\right]_{m,n=1}^p, \quad s>0,
		$$
		belong to $PD_p(X\times Y)$.\ If each $s\in (0,\infty) \mapsto s^{-v_{m,n}((x,y),(x',y'))/2}$ is $\rho$-integrable, Theorem \ref{pp-detemen} implies that the formula
		\begin{eqnarray*}
			K_{m,n}((x,y)\hspace{-3mm}&,&\hspace{-3mm}(x',y')) =  \frac{\Gamma(v_{m,n}((x,y),(x',y'))}{\sqrt{\det G_{m,n}(y,y')}}\\
			&  & \hspace*{-5mm}  \times\mathcal{M}_{v_{m,n}((x,y),(x',y'))}(r(H_{m,n}(x,x')^\intercal
			G_{m,n}(y,y')^{-1} H_{m,n}(x,x'))^{1/2})
		\end{eqnarray*}
		defines a kernel $K((x,y),(x',y'))=[K_{m,n}((x,y),(x',y'))]_{m,n=1}^p$ that belongs to $PD_p(X\times Y)$, as long as the $G_{m,n}$ and the $H_{m,n}$ satisfy the assumptions of the theorem.\
		We could also modify the $P_{m,n}^s$
		by introducing a matrix $[r_{m,n}]_{m,n=1}^p$ with positive entries, by setting
		$$
		P^s_{m,n}((x,y),(x',y')):= \frac{r_{m,n}^{\displaystyle{2v_{m,n}((x,y),(x',y'))}} s^{\displaystyle{-v_{m,n}((x,y),(x', y'))}}}{2^{\displaystyle{2v_{m,n}((x,y),(x',y'))}} },
		$$
		for $x,x'\in X$ and $y,y'\in Y$,
		as long as the kernels
		$$
		((x,y),(x',y')\in (X\times Y)^2\mapsto \left[P_{m,n}^s((x,y),(x',y'))\right]_{m,n=1}^p, \quad s>0,
		$$
		stay in $PD_p(X\times Y)$.\ In this case, the outcome of Theorem \ref{p-detemen} would be that the formula
		\begin{eqnarray*}
			K_{m,n}((x, y)\hspace{-3mm}&,&\hspace{-3mm}(x',y'))  =   \frac{\Gamma(v_{m,n}((x,y),(x',y'))}{\sqrt{\det G_{m,n}(y,y')}}\\
			& & \hspace*{-10mm}\times\mathcal{M}_{v_{m,n}((x,y),(x',y'))}(r_{m,n}(H_{m,n}(x,x')^\intercal  G_{m,n}(y,y')^{-1} H_{m,n}(x,x'))^{1/2})
		\end{eqnarray*}
		defines a kernel $K((x,y),(x',y'))=[K_{m,n}((x,y),(x',y'))]_{m,n=1}^p$ in $PD_p(X\times Y)$, if we keep the assumptions on the $G_{m,n}$ and the $H_{m,n}$ required in the theorem.\
		An specific and simple example in the space-time setting can be produced in analogy with Theorem 1 in \cite{bourotte}: set $Y=\mathbb{R}^d$, $X=\mathbb{R}$,
		$$
		G_{m,n}(y,y')=g(\|y-y'\|^2)I_q,\quad y,y'\in \mathbb{R}^d; m,n=1,\ldots,p,
		$$
		where $g:(0,\infty)\to (0,\infty)$ has a completely monotone derivative and
		$$
		H_{m,n}(x,x')=x-x',\quad x,x' \in \mathbb{R}; m,n=1,\ldots,p.
		$$
		Since $(y,y')\in \mathbb{R}^d \times \mathbb{R}^d \mapsto g(\|y-y'\|^2)$ belongs to $CND_1(\mathbb{R}^d)$ by a result of Micchelli (\cite{micchelli}), it follows that the matrix kernels $
		(y,y')\in Y \times Y \mapsto [u^{\intercal} G_{m,n}(y,y')u]_{m,n=1}^p$, $u \in \mathbb{R}^q$,
		belong to $CND_p(\mathbb{R}^d)$.\ If we put
		$$
		v_{m,n}((x,y),(x',y'))=\frac{v_m+v_n}{2}, \quad x,x'\in \mathbb{R}; y,y'\in Y; m,n=1,\ldots, p,
		$$
		in which each $v_m$ is a positive constant and properly specify  $[r_{m,n}]_{m,n=1}^p$, then for $x,x'\in \mathbb{R}$ and $y,y'\in \mathbb{R}^d$ the formula
		$$
		P^s_{m,n}((x,y),(x',y')):= \frac{r_{m,n}^{\displaystyle{v_m+v_n}} s^{\displaystyle{-(v_m +v_n)/2}}}{2^{\displaystyle{v_m+v_n}} }, \ \ m,n=1,\ldots,p,
		$$
		defines kernels
		$$
		((x,y),(x',y'))\in (\mathbb{R} \times \mathbb{R}^d)^2 \mapsto [P^s_{m,n}((x,y),(x',y'))]_{m,n=1}^p, \quad s>0
		$$
		in $PD_p(\mathbb{R} \times \mathbb{R}^d)$.\ An application of Theorem \ref{pp-detemen} would lead to
		$$
		K_{m,n}((x,y),(x',y'))= \frac{\Gamma((v_m+v_n)/2)}{g(\|y-y'\|^{p/2}}\mathcal{M}_{(v_m+v_n)/2}\left(r_{mn}\frac{\|x-x'\|^2}{g(\|y-y'\|^2)}\right)
		$$
		with $K((x,y),(x',y'))=[K_{m,n}((x,y),(x',y'))]_{m,n=1}^p$ in $PD_p(\mathbb{R} \times \mathbb{R}^d)$.\ We observe that the factor $\Gamma((v_m+v_n)/2)$ can be eliminated as long as we can specify
		$[r_{mn}]_{m,n=1}^p$ in such a way that $[r_{m,n}^{v_m+v_n}/\Gamma((v_m+v_n)/2)]_{m,n=1}^p$ is a positive definite matrix.\
		Theorem 1 in \cite{kleiber} is another construction that fits into Theorem \ref{pp-detemen}.\ Details on that will be left to the readers.
\end{ex}

\begin{ex} The so-called generalized Cauchy function (\cite[p.337]{grad}) is given by
		$$
		\frac{1}{(1+cu^\gamma)^{\nu}}=\frac{c^{-\nu}}{\Gamma(\nu)} \int_0^\infty e^{-s\,u^\gamma} s^{\nu}d\rho(s), \quad u\geq 0,
		$$
		where $c>0$, $\nu>1$, $\gamma\in(0,1]$, and $d\rho(s)=s^{-1}\exp(-\, s/c)$.\ In order to apply Theorem \ref{pp-detemen} we now set
		$\phi(u)=e^{-u^\gamma}$, $u>0$ and
		$$
		P_{m,n}^s((x,y),(x',y')=\left(\frac{s}{c}\right)^{v_m(x,y)+v_n(x',y')}, \quad s>0;\, y,y'\in Y,
		$$
		where $v_m: X \times Y \to (0,\infty)$ is chosen in such a way that each $s\in(0,\infty) \mapsto s^{v_m(x,y)/2}$ is $\rho$-integrable.\ The outcome is that
		\begin{eqnarray*}
			K_{m,n}((x,y),(x',y')) & = & \frac{\Gamma(v_m(x,y)+v_n(x',y'))}{\sqrt{\det G_{m,n}(y,y')}}\\
			& & \times \frac{1}{\left(1+c(H_{m,n}(x,x')^\intercal  G_{m,n}(y,y')^{-1} H_{m,n}(x,x')^\gamma\right)^{v_m(x,y)+v_n(x',y')}},
		\end{eqnarray*}
		defines a kernel $K((x,y),(x',y'))=[K_{m,n}((x,y),(x',y'))]_{m,n=1}^p$ in $PD_p(X\times Y)$, if we keep the assumptions on the $G_{m,n}$ and the $H_{m,n}$ required in the theorem.\
		Arguments similar to those developed in the second half of Example \ref{porcufurado} leads to an example aligned with Theorem 2 in \cite{bourotte}.
\end{ex}

\section{A further extension }

As a final remark let us point an improvement that one can make in all the theorems proved in this paper.\ If for each $m$ and $n$ in $\{1,\ldots, p\}$, $G_{m,n}:Y \times Y \to M_q(\mathbb{R})$ is a matrix function with range containing positive definite matrices only, Theorem \ref{p-detem} justifies the following fact: if the matrix kernels
$$
(y,y')\in Y \times Y \mapsto [u^{\intercal} G_{m,n}(y,y')u]_{m,n=1}^p,\quad u \in \mathbb{R}^q,
$$
belong to $CND_{p}(Y)$, then the kernel $K$ given by
$$
K(y,y')=\left[\frac{1}{\sqrt{\det G_{m,n}(y,y')}}\right]_{m,n=1}^p, \quad y,y' \in Y,
$$
belongs to $PD_p(Y)$.\ Under the same setting, it follows from the Schur Product Theorem that
$$
K_l(y,y')=\left[\frac{1}{[\det G_{m,n}(y,y')]^{l/2}}\right]_{m,n=1}^p, \quad y,y' \in Y.
$$
belongs to $PD_p(Y)$ whenever $l\in \{1,2,\ldots\}$.\ In particular, we can introduce the same power $l/2$
in the assertions of all the theorems proved in the paper.

%
%

\vspace*{1cm}

\noindent V. A. Menegatto \\
Departamento de
Matem\'atica - ICMC-USP - S\~ao Carlos\\
Caixa Postal 668\\
13560-970, S\~ao Carlos SP, Brazil\\
E-mail: menegatt@gmail.com
\\
\\
C. P. Oliveira\\
 Instituto de Matem\'{a}tica e Computa\c c\~{a}o - UNIFEI\\
 Av. BPS, 1303, Pinheirinho \\
 37500-903, Itajub\'{a} MG, 	Brazil.\\
E-mail: oliveira@unifei.edu.br

\end{document}